\documentclass[a4paper,10pt,leqno]{amsart}
\usepackage{amsmath,amsfonts,amsthm,amssymb,dsfont}
\usepackage[alphabetic]{amsrefs}
\usepackage[OT4]{fontenc}
\usepackage{enumerate}
\usepackage{mathrsfs}
\usepackage{mathtools}
\usepackage{enumerate}
\usepackage{float}
\usepackage[dvipsnames]{xcolor}

\newtheorem{thm}{Theorem}[section]
\newtheorem{lem}[thm]{Lemma}
\newtheorem{prop}[thm]{Proposition}

\newtheorem{theoremalph}{Theorem}

\theoremstyle{definition}
\newtheorem{rem}[thm]{Remark}
\newtheorem{defin}[thm]{Definition}

\newcommand{\Z}{\mathbb{Z}}
\newcommand{\R}{\mathbb{R}}

\makeatletter
\def\paragraph{\@startsection{paragraph}{4}%
  \z@\z@{-\fontdimen2\font}%
  {\normalfont\bfseries}}
\makeatother

\begin{document}

\title[Abelian subgroups of two-dimensional Artin groups]{Abelian subgroups of two-dimensional Artin groups}

\author[A.~Martin]{Alexandre Martin$^{\dag*}$}
           \address{Department of Mathematics and the Maxwell Institute for Mathematical Sciences,
           Heriot-Watt University,
           Riccarton,
           EH14 4AS Edinburgh, United Kingdom}
           \email{alexandre.martin@hw.ac.uk}
           \thanks{$\dag$ Partially supported by EPSRC New Investigator Award EP/S010963/1.}

\author[P.~Przytycki]{Piotr Przytycki$^{\ddagger*}$}
\address{
Department of Mathematics and Statistics,
McGill University,
Burnside Hall,
805 Sherbrooke Street West,
Montreal, QC,
H3A 0B9, Canada}

\email{piotr.przytycki@mcgill.ca}
\thanks{$\ddagger$ Partially supported by NSERC, FRQNT, and National Science Centre, Poland UMO-2018/30/M/ST1/00668.}
\thanks{$*$ This work was partially supported by
the grant 346300 for IMPAN from the Simons Foundation and the matching
2015--2019 Polish MNiSW fund}

\maketitle

\begin{abstract}
We classify abelian subgroups of two-dimensional Artin groups.
\end{abstract}

\section{Introduction}

Let $S$ be a finite set and for all $s\neq t\in S$ let $m_{st}=m_{ts}\in \{2,3,\ldots, \infty\}$.
The associated \emph{Artin group} $A_S$ is given by generators and relations:
$$ A_S = \langle S \ | \ \underbrace{sts\cdots}_{m_{st}}=\underbrace{tst\cdots}_{m_{st}}\rangle.$$ 
Assume that $A_S$ is \textit{two-dimensional}, that is, for all $s, t, r \in S$, we have
$$\frac{1}{m_{st}} + \frac{1}{m_{tr}}+\frac{1}{m_{sr}} \leq 1.$$
We say that $A_S$ is of \textit{hyperbolic type} if the inequality above is strict for all $s, t, r \in S$.

In \cite[Thm~D]{MP} we classified explicitly all virtually abelian subgroups of $A_S$ of hyperbolic type. In this article we extend this, using different techniques, to all two-dimensional $A_S$. Charney and Davis \cite[Thm~B]{CD} showed that two-dimensional Artin groups satisfy the $K(\pi,1)$-conjecture, in particular they are torsion-free and of cohomological dimension $2$. Thus all noncyclic virtually abelian subgroups of~$A_S$ are virtually~$\Z^2$. The first step of our classification involves the \textit{modified Deligne complex} $\Phi$ (see Section~\ref{sec:complex}) introduced by Charney and Davis \cite{CD}, and generalising a construction of Deligne for Artin groups of spherical type \cite{Del}.

\begin{theoremalph}
\label{thm:classify}
Let $A_S$ be a two-dimensional Artin group, and let $H$ be a subgroup of $A_S$ that is virtually
$\Z^2$. Then:
\begin{enumerate}[(i)]
\item $H$ is contained in the stabiliser of a vertex of $\Phi$, or
\item $H$ is contained in the stabiliser of a standard tree of $\Phi$, or
\item $H$ acts properly on a Euclidean plane isometrically embedded in $\Phi$.
\end{enumerate}
\end{theoremalph}

The stabilisers of the vertices of $\Phi$, appearing in (i), are cyclic or conjugates of the dihedral Artin groups $A_{st}$ (see Section~\ref{sec:complex}), which are virtually $\Z\times F$ for a free group $F$, see e.g.\ \cite[Lem~4.3(i)]{HJP}.
The stabilisers of the standard trees of $\Phi$, appearing in (ii), are also $\Z\times F$ \cite[Lem~4.5]{MP} and were described more explicitly in \cite[Rm~4.6]{MP}.

In the second step of our classification, we will list all $H\cong \Z^2$ satisfying (iii). Note that the statement might seem daunting, but in fact it arises in a straightforward way from reading off the labels of the Euclidean planes obtained in the course of the proof.

We use the following notation. Let $S$ be an alphabet. If $s\in S$, then $s^\bullet$ denotes the language (i.e.\ set of words) of form $s^n$ for $n\in \Z-\{0\}$. We treat a letter $s\in S^\pm$ as a language consisting of a single word. If $\mathcal L, \mathcal L'$ are languages, then $\mathcal L\mathcal L'$ denotes the language of words of the form $ww'$ where $w\in \mathcal L, w'\in \mathcal L'$. If $\mathcal L$ is a language, then $\mathcal L^*$ denotes the union of the languages $\mathcal L^n$ for $n\geq 1$.

\begin{theoremalph}
\label{thm:main} Let $A_S$ be a two-dimensional Artin group.
Suppose that $\Z^2\subset A_S$ acts properly on a Euclidean plane isometrically embedded in $\Phi$. Then $\Z^2$ is conjugated into:
\begin{enumerate}[(a)]
\item $\langle w, w'\rangle$, where $w\in A_T,w'\in A_{T'}$ and $m_{tt'}=2$ for all $t\in T,t'\in T'$, for some disjoint $T,T'\subset S$, \\
    or $\Z^2$ is conjugated into one of the following, where $s,t,r\in S$:
\item $\langle strstr, w\rangle$, where $w\in (t^\bullet str)^*$ and $m_{st}=m_{tr}=m_{sr}=3$.
\item $\langle strt, w\rangle$, where $w\in (r^\bullet t^{-1}s^\bullet t)^*$ and $m_{st}=m_{tr}=4,m_{sr}=2$.
\item $\langle stsrtr, w\rangle$, where $w\in (t^\bullet str)^*$ and $m_{st}=m_{tr}=4,m_{sr}=2$.
\item $\langle ststsrtstr, w\rangle$, where $w\in (t^\bullet ststr)^*$ and $m_{st}=6,m_{tr}=3,m_{sr}=2$.
\item $\langle tststr, w\rangle$, where $w\in (s^\bullet tstrt^{-1})^*$ and $m_{st}=6,m_{tr}=3,m_{sr}=2$.
\end{enumerate}
\end{theoremalph}

It is easy to check directly that the above groups are indeed abelian. Since $A_S$ is torsion-free \cite[Thm~B]{CD}, the only other subgroups of $A_S$ that are virtually~$\Z^2$ are isomorphic to the fundamental group of the Klein bottle. They can be also classified, see Remark~\ref{rem:klein}.

In the proof of Theorem~\ref{thm:main} we will describe in detail the Euclidean planes in~$\Phi$ stabilised by $\Z^2\subset A_S$. Huang and Osajda established properties of arbitrary quasiflats in the Cayley complex of $A_S$, and one can find similarities between our results and \cite[\S5.1—5.2 and~Prop~8.3]{OH2}.

\smallskip

\paragraph{Organisation of the article.} In Section~\ref{sec:complex} we describe the modified Deligne complex $\Phi$ of Charney and Davis and we prove Theorem~\ref{thm:classify}. In Section~\ref{sec:preliminaries} we prove a lemma on dihedral Artin groups fitting in the framework of \cite{AS}. In Section~\ref{sec:pola} we introduce a \emph{polarisation} method for studying Euclidean planes in~$\Phi$. We finish with the classification of admissible polarisations and the proof of Theorem~\ref{thm:main} in Section~\ref{sec:main}.

\section{Modified Deligne complex}
\label{sec:complex}
Let $A_S$ be a two-dimensional Artin group. For $s,t\in S$ satisfying $m_{st}<\infty$, let $A_{st}$ be the \emph{dihedral} Artin group $A_{S'}$ with $S'=\{s,t\}$ and exponent $m_{st}$. For $s\in S$, let $A_s=\Z$.

Let $K$ be the following simplicial complex. The vertices of $K$ correspond to subsets $T\subseteq S$ satisfying $|T|\leq 2$ and, in the case where $|T|=2$ with $T=\{s,t\}$, satisfying $m_{st}<\infty$. We call $T$ the \emph{type} of its corresponding vertex. Vertices of types $T,T'$ are connected by an edge of $K$, if we have $T\subsetneq T'$ or vice versa. Similarly, three vertices span a triangle of $K$, if they have types $\emptyset, \{s\}, \{s,t\}$ for some $s,t\in S$.

We give $K$ the following structure of a simple complex of groups $\mathcal K$ (see \cite[\S II.12]{BH} for background). The vertex groups are trivial, $A_s$, or $A_{st}$, when the vertex is of type $\emptyset,\{s\},\{s,t\}$, respectively. For an edge joining a vertex of type $\{s\}$ to a vertex of type $\{s,t\}$, its edge group is $A_s$; all other edge groups and all triangle groups are trivial. All inclusion maps are the obvious ones. It follows directly from the definitions that $A_S$ is the fundamental group of $\mathcal K$.

We equip each triangle of $K$ with the \emph{Moussong metric} of an Euclidean triangle of angles $\frac{\pi}{2m_{st}}, \frac{\pi}{2}, \frac{(m_{st}-1)\pi}{2m_{st}}$ at the vertices of types $\{s,t\},\{s\},\emptyset$, respectively. As explained in \cite[\S3]{MP},
the local developments of $\mathcal K$ are $\mathrm{CAT}(0)$ and hence $\mathcal K$ is strictly developable and its development $\Phi$ exists and is $\mathrm{CAT}(0)$. See \cite{CD} for a detailed proof. We call $\Phi$ with the Moussong metric the \emph{modified Deligne complex}. In particular all $A_s$ and $A_{st}$ with $m_{st}<\infty$ map injectively into $A_S$ (which follows also from \cite[Thm 4.13]{L}). Vertices of $\Phi$ inherit types from the types of the vertices of $K$.

Let $r\in S$ and let $T$ be the fixed-point set in $\Phi$ of $r$. Note that since $A_S$ acts on~$\Phi$ without inversions, $T$ is a subcomplex of~$\Phi$. Since the stabilisers of the triangles of~$\Phi$ are trivial, we have that $T$ is a graph. Since $\Phi$ is $\mathrm{CAT}(0)$, $T$ is convex and thus it is a tree. Thus we call a \textit{standard tree} the fixed-point set in $\Phi$ of a conjugate of a generator $r\in S$ of~$A_S$.

\begin{rem}[{\cite[Rm~4.4]{MP}}]
\label{rem:trees unique} Each edge of $\Phi$ belongs to at most one standard tree.
\end{rem}

\begin{proof}[{Proof of Theorem~\ref{thm:classify}}]
Let $\Gamma\subset H$ be a finite index normal subgroup isomorphic to~$\Z^2$. By \cite{B}, $\Gamma$ acts on $\Phi$ by semi-simple isometries.
Let $\mathrm{Min}(\Gamma)=\bigcap_{\gamma\in \Gamma}\mathrm{Min}(\gamma)$, where $\mathrm{Min}(\gamma)$ is the Minset of $\gamma$ in $\Phi$. By a variant of the Flat Torus Theorem not requiring properness \cite[Thm~II.7.20(1)]{BH}, $\mathrm{Min}(\Gamma)$ is nonempty.
By \cite[Thm~II.7.20(4)]{BH} we have that $H$ stabilises $\mathrm{Min}(\Gamma)$.

Suppose first that each element of $\Gamma$ fixes a point of $\Phi$. Then $\Gamma$ acts trivially on $\mathrm{Min}(\Gamma)$. By the fixed-point theorem \cite[Thm~II.2.8(1)]{BH} the finite group $H/\Gamma$ fixes a point of $\mathrm{Min}(\Gamma)$, and since the action is without inversions, we can take this point to be a vertex as required in (i).

Secondly, suppose that $\Gamma$ has both an element $\gamma$ that fixes a point of $\Phi$ and an element that is loxodromic. Then $\mathrm{Min}(\Gamma)$ is not a single point, so it contains an edge $e$. Since $\mathrm{Min}(\Gamma)\subset\mathrm{Fix}(\gamma)$, we have that $\gamma$ fixes $e$. Thus $\gamma$ is a conjugate of an element of $S$ and so $\mathrm{Min}(\Gamma)$ is contained in a standard tree $T$. For any $h\in H$ we have that the intersection $h(T)\cap T$ contains $\mathrm{Min}(\Gamma)\supset e$ and thus by Remark~\ref{rem:trees unique} we have $h\in \mathrm{Stab}(T)$, as required in (ii).

Finally, suppose that all elements of $\Gamma$ are loxodromic. By \cite[Thm~II.7.20(1,4)]{BH} we have $\mathrm{Min}(\Gamma)=Y\times\R^n$ with $H$ preserving the product structure and $\Gamma$ acting trivially on $Y$. As before $H/\Gamma$ fixes a point of $Y$ and so $H$ stabilises $\R^n$ isometrically embedded in $\Phi$. By \cite[Thm~II.7.20(2)]{BH}, we have $n\leq 2$, but since $H$ acts by simplicial isometries, we have $n=2$ and the action is proper, as required in (iii).
\end{proof}

\section{Girth lemma}
\label{sec:preliminaries}

\begin{lem}
\label{lem:AS2}
Let $S=\{s,t\}$ with $m_{st}\geq 3$. A word with $2m$ syllables (i.e.\ of form $s^{i_1}t^{j_1}\cdots s^{i_m}t^{j_m}$ with all $i_k,j_k\in \Z-\{0\}$) is trivial in $A_S$ if and only if up to interchanging $s$ with $t$, and a cyclic permutation, it is of the form:
\begin{itemize}
\item $s^k\underbrace{t\cdots s}_{m-1}t^{-k}\underbrace{s^{-1}\cdots t^{-1}}_{m-1}$ for $m$ odd,
\item
$s^k\underbrace{ts\cdots t}_{m-1}s^{-k}\underbrace{t^{-1}s^{-1}\cdots t^{-1}}_{m-1}$ for $m$ even,
\end{itemize}
where $k\in \Z-\{0\}$.
\end{lem}

\begin{proof}
The `if' part follows immediately from Figure~\ref{Figure_lox_0}.
We prove the `only if' part by induction on the size of any reduced (van~Kampen) diagram $M$ of the word $w$ in question, where we prove the stronger assertion that, up to interchanging $s$ with~$t$, $M$ is as in Figure~\ref{Figure_lox_0}.

\begin{figure}[H]
\begin{center}
 \scalebox{1}{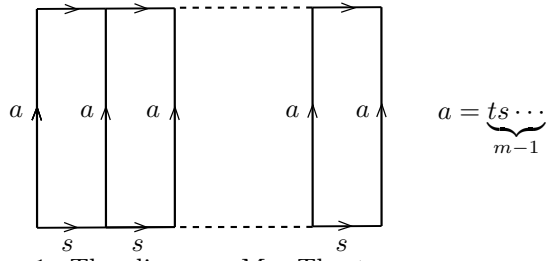}
\caption{The diagram $M$. The top arrows are either all labelled~$s$ or all labelled $t$, depending on the parity of $m$.}
\label{Figure_lox_0}
\end{center}
\end{figure}

We use the vocabulary from~\cite{AS}, where the $2$-cells of $M$ are called \emph{regions} and the \emph{interior degree} $i(D)$ of a region is the number of interior edges of $\partial D$ (after forgetting vertices of valence $2$). For example the two extreme regions in Figure~\ref{Figure_lox_0} have interior degree $1$. A region $D$ is a \emph{simple boundary region} if $\partial D\cap \partial M$ is nonempty, and $M-\overline D$ is connected. For example, the two extreme regions in Figure~\ref{Figure_lox_0} are simple boundary regions, but the remaining ones are not. A \emph{singleton strip} is a simple boundary region with $i(D)\leq 1$. A \emph{compound strip} is a subdiagram $R$ of $M$ consisting of regions $D_1,\ldots, D_n,$ with $n\geq 2$, with $D_{k-1}\cap D_k$ a single interior edge of $R$ (after forgetting vertices of valence $2$), satisfying $i(D_1)=i(D_n)=2,i(D_k)=3$ for $1<k<n$ and $M-R$ connected.

Let $R$ be a strip of $M$ with boundary labelled by $rb$, where $r$ labels $\partial R\cap \partial M$ and so $w=rw'$. Assume also that $R$ shares no regions with some other strip (such a pair of strips exists by \cite[Lem~2]{AS}). By \cite[Lem~5]{AS}, we have that the syllable lengths satisfy $||r||\geq ||b||+2$ and so by \cite[Lem~6]{AS}, we have
$||r||\geq m+1$, hence $||w'||\leq m+1$.
In fact, since the outside boundary of the other strip has also syllable length $\geq m+1$, we have $||r||=m+1$, and hence $||b||=m-1$. Let $M'$ be the diagram with boundary labelled by $b^{-1}w'$ obtained from $M$ by removing~$R$. By the induction hypothesis, $M'$ is as in Figure~\ref{Figure_lox_0}. If $R$ is a singleton strip, then there is only one way of gluing $R$ to $M'$ to obtain $||w||=2m$ and it is as in Figure~\ref{Figure_lox_0}. If $R$ is a compound strip, then by the induction hypothesis $R$ is also as in Figure~\ref{Figure_lox_0}. Moreover, since all regions of $R$ share exactly one edge with $M'$, up to interchanging $s$ with $t$, and/or $b$ with $b^{-1}$, we have $b=s^k\underbrace{ts\cdots }_{m-2}$ or $b=\underbrace{\cdots st}_{m-2}s^k$, where $k\geq 2$. Since $m>2$, we have that $b$ cannot be a subword of the boundary word of $M'$, unless $M'$ is a mirror copy of~$R$, contradiction.
\end{proof}

\begin{rem}
\label{rem:AS}
Let $S=\{s,t\}$ with $m_{st}=2$. A word with $4$ syllables is trivial in $A_S=\Z^2$ if and only if up to interchanging $s$ with $t$ it is of the form $s^kt^ls^{-k}t^{-l}$, where $k,l\in \Z-\{0\}$.
\end{rem}

\section{Polarisation}
\label{sec:pola}

\begin{defin}
Let $F$ be a Euclidean plane isometrically embedded in $\Phi$. Then for each vertex $v$ in $F$ of type $\{s,t\}$ there are exactly $4m_{st}$ triangles in $F$ incident to $v$. We assemble them into regular $2m_{st}$-gons, and call this complex the \emph{tiling of~$F$}. We say that a cell of this tiling has \emph{type $T$} if its barycentre in $\Phi$ has type $T$.
\end{defin}

For a Coxeter group $W$, let $\Sigma$ denote its \emph{Davis complex}, i.e.\ the complex obtained from the standard Cayley graph by adding $k$-cells
corresponding to cosets of finite $\langle T\rangle$ for $T\subset S$ of size $k$. For example for $W$ the triangle Coxeter group with exponents $\{3,3,3\}$, the complex $\Sigma$ is the tiling of the Euclidean plane by regular hexagons.

\begin{lem}
\label{lem:tiling}
Let $F$ be a Euclidean plane isometrically embedded in $\Phi$. Then the tiling of $F$ is either the standard square tiling, or the one of the Davis complex $\Sigma$ for $W$, where $W$ is the triangle Coxeter group with exponents $\{3,3,3\},\{2,4,4\}$ or $\{2,3,6\}$.
\end{lem}

\begin{proof}
The $2$-cells of the tiling are regular polygons with even numbers of sides, hence their angles lie in $[\frac{\pi}{2},\pi)$.
If there is a vertex $v$ of $F$ incident to four $2$-cells, then all these $2$-cells are squares. Consequently, any vertex of $F$ adjacent to $v$ is incident to at least two squares, and thus to exactly four squares. Then, since the $1$-skeleton of $F$ is connected, the tiling of $F$ is the standard square tiling.

If $v$ is incident to three $2$-cells, which are $2m,2m',2m''$-gons, then since $\frac{1}{m}+\frac{1}{m'}+\frac{1}{m''}=1$, we have $\{m,m',m''\}=\{3,3,3\},\{2,4,4\}$ or $\{2,3,6\}$. Moreover, a vertex $u$ of $F$ adjacent to $v$ is incident to two of these three $2$-cells, and this implies that the third $2$-cell incident to $u$ has the same size as the one incident to $v$. This determines uniquely the tiling of $F$ as the one of $\Sigma$.
\end{proof}

Henceforth, let $\Sigma$ be the Davis complex for $W$, where $W$ is the triangle Coxeter group with exponents $\{3,3,3\},\{2,4,4\}$ or $\{2,3,6\}$.

\begin{lem}
\label{lem:tiling_types}
Suppose $\Sigma$ is the tiling of a Euclidean plane isometrically embedded in $\Phi$.
Then the natural action of $W$ on $\Sigma$ preserves the edge types coming from $\Phi$.
\end{lem}

In particular, $W=W_T$ for some $T\subset S$ with $|T|=3$.

\begin{proof}
Chose a vertex $v$ of $\Sigma$, and let $\{s\},\{t\},\{r\}$ be the types of edges incident to~$v$. Let $u$ be a vertex of $\Sigma$ adjacent to $v$, say along an edge $e$ of type $\{r\}$. Hence the $2$-cells in $\Sigma$ incident to $e$ have types $\{s,r\}$ and $\{t,r\}$. Consequently, the types of the remaining two edges incident to $u$ are also $\{s\}$ and $\{t\}$, and in such a way that the reflection of $\Sigma$ interchanging the endpoints of $e$ preserves the types of these edges. This determines uniquely the types of the edges of $\Sigma$, and guarantees that they are preserved by $W$.
\end{proof}

\begin{defin}
A \emph{polarisation} of $\Sigma$ is a choice of a longest diagonal $l(\sigma)$ in each $2$-cell $\sigma$ of $\Sigma$. A polarisation is \emph{admissible} if
every vertex of $\Sigma$ belongs to exactly one $l(\sigma)$.
\end{defin}

\begin{defin}
\label{def:induced_polarisation}
Suppose $\Z^2\subset A_S$ acts properly and cocompactly on $\Sigma\subset \Phi$. For an edge $e$ of type $\{s\}$ in $\Sigma$, its vertices correspond to elements $g,gs^k\in A_S$ for $k>0$. We direct $e$ from $g$ to $gs^k$. By Lemma~\ref{lem:AS2}, the boundary of each $2$-cell $\sigma$ is subdivided into two directed paths joining two opposite vertices. The \emph{induced polarisation} of $\Sigma$ assigns to each $\sigma$ the longest diagonal $l(\sigma)$ joining these two vertices.
\end{defin}

\begin{lem}
\label{lem:adm}
An induced polarisation is admissible.
\end{lem}
\begin{proof}
\textbf{Step 1.} \emph{For each vertex $v$ of $\Sigma$, there is at most one $l(\sigma)$ containing $v$.}
\smallskip

Indeed, suppose that we have $v\in l(\sigma), l(\tau)$. Without loss of generality suppose that the edge $e=\sigma\cap \tau$ is directed from $v$. Then the other two edges incident to~$v$ are also directed from $v$. We will now prove by induction on the distance from~$v$ that each edge of $\Sigma$ is oriented from its vertex closer to $v$ to its vertex farther from~$v$ in the $1$-skeleton $\Sigma^1$ (they cannot be at equal distance since $\Sigma^1$ is bipartite).

For the induction step, suppose we have already proved the induction hypothesis for all edges closer to $v$ than an edge $uu'$, where $u'$ is closer to $v$ than $u$. Let $u''$ be the first vertex on a geodesic from $u'$ to $v$ in $\Sigma^1$. Let $\sigma$ be the $2$-cell containing the path $uu'u''$. By \cite[Thms~2.10 and~2.16]{R}, $\sigma$ has two opposite vertices $u_0$ closest to $v$ and $u_\mathrm{max}$ farthest from $v$. By the induction hypothesis, the edge $u'u''$ is oriented from $u''$ to $u'$, and both edges of $\sigma$ incident to $u_0$ are oriented from $u_0$. Thus if the edge $uu'$ was oriented to $u'$ we would have that $u'$ is opposite to $u_0$, so $u'=u_\mathrm{max}$, contradiction. This finishes the induction step.

As a consequence, $v$ is the unique vertex of $\Sigma$ with all edges incident to $v$ oriented from $v$. This contradicts the cocompactness of the action of $\Z^2$ on $\Sigma$ and proves Step~1.

\smallskip
\noindent\textbf{Step 2.} \emph{For each $v$ there is at least one $l(\sigma)$ containing $v$.}
\smallskip

Among the edges incident to $v$ there are at least two edges directed from $v$ or at least two edges directed to $v$. The $2$-cell $\sigma$ containing such two edges satisfies $l(\sigma)\ni v$.
\end{proof}

\section{Classification}
\label{sec:main}

\begin{prop}
\label{prop:333}
Let $\Sigma$ be the Davis complex for $W_T$ the triangle Coxeter group with exponents $\{3,3,3\}$ with an admissible polarisation $l$. Then there is an edge $e$ such that each hexagon $\gamma$ of $\Sigma$ satisfies
\begin{description}
\item[$\clubsuit$]
the diagonal $l(\gamma)$ has endpoints on edges of $\gamma$ parallel to $e$.
\end{description}
\end{prop}

Note that if the conclusion of Proposition~\ref{prop:333} holds, then the translation $\rho$ mapping one hexagon containing $e$ to the other preserves $l$.

\begin{rem} It is easy to prove the converse, i.e.\ that if each $l(\gamma)$ has endpoints on edges of $\gamma$ parallel to $e$, and if $l$ is $\rho$-invariant, then $l$ is admissible. This can be used to classify all admissible polarisations, but we will not need it.
\end{rem}

To prove Proposition~\ref{prop:333} we need the following reduction.

\begin{lem}
\label{lem:strips} Let $e$ be an edge and $\rho$ a translation mapping one hexagon containing~$e$ to the other. If $\clubsuit$ holds for all hexagons~$\gamma$ in one $\rho$-orbit, then it holds for all~$\gamma$.
\end{lem}
\begin{proof}
Suppose that $\clubsuit$ holds for all hexagons $\gamma$ in the $\rho$-orbit of a hexagon $\sigma$. Let $\tau$ be a hexagon adjacent to two of them, say to $\sigma$ and $\rho(\sigma)$. Let $v=\sigma\cap\rho(\sigma)\cap\tau$. Since $\clubsuit$ holds for $\gamma=\sigma$ and $\gamma=\rho(\sigma)$, by the admissibility of $l$, $v$ belongs to one of $l(\sigma), l(\rho(\sigma))$. Thus $v\notin l(\tau)$ and hence $\clubsuit$ holds for $\gamma=\tau$. Proceeding inductively, by the connectivity of $\Sigma$, we obtain $\clubsuit$ for all $\gamma$.
\end{proof}

\begin{proof}[Proof of Proposition~\ref{prop:333}]
\noindent\textbf{Case 1.} There are adjacent hexagons $\sigma,\tau$ with non-parallel $l(\sigma),l(\tau)$.
\smallskip

Let $f=\sigma\cap \tau$. Without loss of generality $l(\sigma)\cap f=\emptyset, l(\tau)\cap f \neq\emptyset$. Let $v$ be the vertex of $f$ outside $l(\tau)$. By the admissibility of $l$, $v$ is contained in $l(\sigma')$ for the third hexagon $\sigma'$ incident to $v$. Hence $\clubsuit$ holds for $e=\sigma\cap\sigma'$ and $\gamma=\sigma, \sigma', \tau$
(see Figure~\ref{Figure_lox_1}). Let $\rho$ be the translation mapping $\sigma$ to~$\sigma'$.
Replacing the pair $\sigma, \tau$ with $\tau, \sigma'$ and repeating inductively the argument shows that
$\clubsuit$ holds for $\gamma=\rho^n(\sigma), \rho^n(\tau)$ for all $n>0$ (note that $e$ gets replaced by parallel edges in this procedure).

\begin{figure}[H]
\begin{center}
 \scalebox{1}{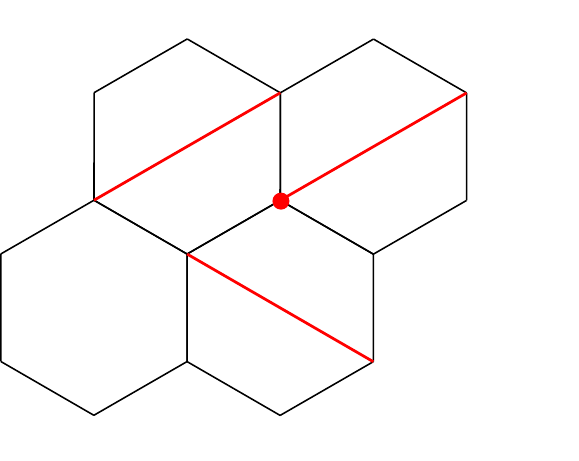}
\caption{}
\label{Figure_lox_1}
\end{center}
\end{figure}

Furthermore, by the admissibility of $l$, since $l(\rho^{-1}(\tau))$ is disjoint from $l(\sigma)$ and~$l(\tau)$, it leaves us only one choice for $l(\rho^{-1}(\tau))$, and it satisfies $\clubsuit$ for $\gamma=\rho^{-1}(\tau)$. Replacing the pair $\sigma, \tau$ with $\rho^{-1}(\tau),\sigma$ and repeating inductively the argument shows that $\clubsuit$ holds for $\gamma=\rho^{-n}(\sigma), \rho^{-n}(\tau)$ for all $n>0$. It remains to apply Lemma~\ref{lem:strips}.

\smallskip
\noindent\textbf{Case 2.} All the $l(\sigma)$ are parallel.
\smallskip

In this case it suffices to take any edge $e$ intersecting some $l(\sigma)$.
\end{proof}

\begin{prop}
\label{prop:244}
Let $\Sigma$ be the Davis complex for $W_T$ the triangle Coxeter group with exponents $\{2,4,4\}$ with an admissible polarisation $l$. Then
there is an edge $e$
such that each octagon $\gamma$ of $\Sigma$ satisfies
\begin{description}
\item[$\diamondsuit$]
the diagonal $l(\gamma)$ has endpoints on edges of $\gamma$ parallel to $e$.
\end{description}
\end{prop}

An edge $e$ of $\Sigma$ lies either in two octagons $\sigma, \sigma'$ or there is a square with two parallel edges $e,e'$ in octagons $\sigma, \sigma'$. The translation of $\Sigma$ mapping $\sigma$ to $\sigma'$ is called an \emph{$e$-translation}.
Note that if the conclusion of Proposition~\ref{prop:244} holds, then an $e$-translation preserves $l$.
\begin{lem}
\label{lem:strips2}
Let $e$ be an edge and $\rho$ an $e$-translation. If $\diamondsuit$ holds for all octagons~$\gamma$ in one $\rho$-orbit, then it holds for all $\gamma$.
\end{lem}
\begin{proof}
Suppose that $\diamondsuit$ holds for all octagons $\gamma$ in the $\rho$-orbit of an octagon $\sigma$. We can assume $e\subset \sigma$. Suppose first that $e$ lies in another octagon $\sigma'$. Then let $\tau$ be an octagon outside the $\rho$-orbit of $\sigma$ adjacent to some $\rho^k(\sigma)$, say $\sigma$. Let $\square, \rho(\square)$ be the two squares adjacent to both $\sigma$ and $\tau$ (see Figure \ref{Figure_lox_2}). By the admissibility of~$l$, we have that $l(\square),l(\rho(\square))$ contain the two vertices of $\sigma\cap \tau$. Consequently, $l(\tau)$ intersects the edge $\tau\cap\rho(\tau)$, and so $\gamma=\tau$ satisfies $\diamondsuit$. It is easy to extend this to all the octagons $\gamma$.

\begin{figure}[H]
\begin{center}
 \scalebox{1}{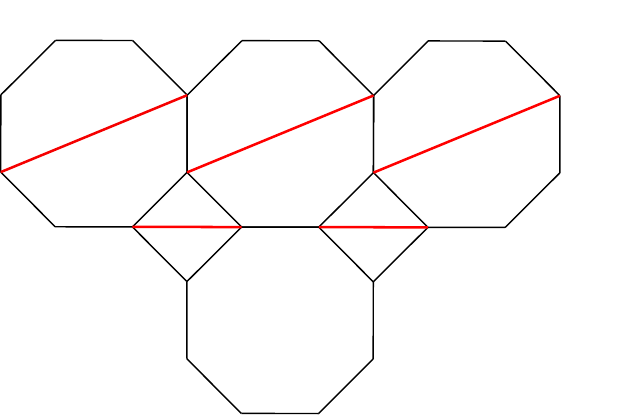}
\caption{}
\label{Figure_lox_2}
\end{center}
\end{figure}

It remains to consider the case where there is a square with two parallel edges $e,e'$ in octagons $\sigma, \sigma'$. Let $\tau$ be an octagon adjacent to two of them, say to $\sigma$ and~$\rho(\sigma)$. Let $v=e\cap \tau, x=e'\cap \tau$. Since $\diamondsuit$ holds for $\gamma=\sigma,\sigma'$, each of $v,x$ lies in one of $l(\sigma), l(\square), l(\sigma')$. Thus by the admissibility of $l$, we have $v,x\notin l(\tau)$. Any of the two remaining choices for $l(\tau)$ satisfy $\diamondsuit$ for $\gamma=\tau$. It is again easy to extend this to all the octagons $\gamma$.
\end{proof}

\begin{proof}[Proof of Proposition~\ref{prop:244}]
Note that we fall in one of the following two cases.

\smallskip
\noindent\textbf{Case 1.} There is an edge $f$ in octagons $\sigma,\tau$ with $l(\sigma)\cap f= \emptyset, l(\tau)\cap f \neq\emptyset$.
\smallskip

Let $v$ be the vertex of $f$ distinct from $u=l(\tau)\cap f$. By the admissibility of $l$, the vertex $v$ is contained in $l(\square)$ for the square $\square$ incident to $v$. Let $x$ be the vertex in $\tau\cap\square$ distinct from $v$, and let $\sigma'$ be the octagon incident to $x$ distinct from $\tau$. By the admissibility of $l$, the vertex $x$ is contained in $l(\sigma')$. Hence $\diamondsuit$ holds for $e=\sigma\cap\square$ and $\gamma=\sigma', \tau$
(see Figure \ref{Figure_lox_3}). Let $\rho$ be the translation mapping $\sigma$ to $\sigma'$.

\begin{figure}[H]
\begin{center}
 \scalebox{1}{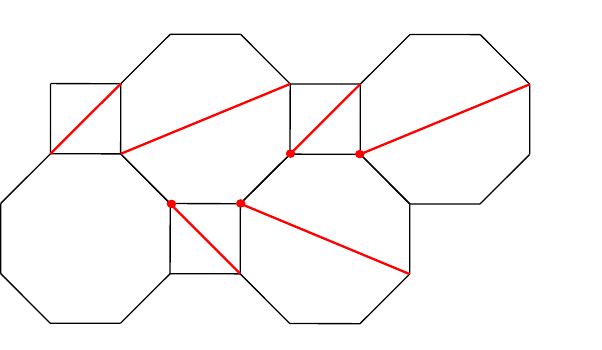}
\caption{}
\label{Figure_lox_3}
\end{center}
\end{figure}

Now let $\blacksquare$ be the square incident to $u$, and let $z$ be the vertex in $\sigma\cap\blacksquare$ distinct from $u$. Note that by the admissibility of $l$, we have $z\in l(\blacksquare)$, and consequently $\rho^{-1}(x)\in l(\sigma)$ and $\rho^{-1}(v)\in l(\rho^{-1}(\square))$. Hence $l(\rho^{-1}(\tau))$ cannot contain neither $z$, nor $\rho^{-1}(v)$, nor $\rho^{-1}(x)\in l(\sigma)$. There is only one remaining choice for $l(\rho^{-1}(\tau))$, and it satisfies $\diamondsuit$ for $\gamma=\rho^{-1}(\tau)$.

We can now argue exactly as in the proof of Proposition~\ref{prop:333} that that $\diamondsuit$ holds for $\gamma=\rho^{n}(\sigma), \rho^{n}(\tau)$ for all $n\in \Z$. It then remains to apply Lemma~\ref{lem:strips2}.

\smallskip
\noindent\textbf{Case 2.} For each edge $e$ in octagons $\sigma,\sigma'$ with $l(\sigma)\cap e\neq \emptyset$ we have $l(\sigma')\cap e \neq\emptyset$.

\smallskip

Let $\sigma$ be any octagon and $e$ an edge contained in another octagon $\sigma'$ and intersecting $l(\sigma)$.
Let $\rho$ be the translation mapping $\sigma$ to $\sigma'$. One can show inductively that $\diamondsuit$ holds for octagons $\gamma=\rho^n(\sigma)$ for all $n\in \Z$. It then remains to apply Lemma~\ref{lem:strips2}.
\end{proof}

Note that for $l$ satisfying $\diamondsuit$ for all octagons $\gamma$, the values of $l$ on octagons determine its values on squares.

\begin{prop}
\label{prop:236}
Let $\Sigma$ be the Davis complex for $W_T$ the triangle Coxeter group with exponents $\{2,3,6\}$ with an admissible polarisation $l$. Then
there is an edge $e$ such that
such that each $12$-gon $\gamma$ of $\Sigma$ satisfies
\begin{description}
\item[$\heartsuit$]
the diagonal $l(\gamma)$ has endpoints on edges of $\gamma$ parallel to $e$.
\end{description}
\end{prop}

Let $e$ be an edge. An \emph{$e$-translation} is the translation of $\Sigma$ mapping $\sigma$ to $\sigma'$ in of the two following configurations. In the first configuration we have a square with two parallel edges $e,e'$ in 12-gons $\sigma, \sigma'$.
In the second configuration we have four parallel edges $e,e',e'',e'''$ such that $e', e''$ lie in a square, $e,e'$ in a hexagon~$\phi$ and $e'',e'''$ in another hexagon, and we consider 12-gons $\sigma\supset e, \sigma'\supset e'''$. Again, if the conclusion of Proposition~\ref{prop:236} holds, then an $e$-translation preserves $l$. To see this in the configuration with hexagons it suffices to observe that $l(\phi)$ (and similarly for the other hexagon) is not parallel to~$e$: otherwise $l(\phi)$ would intersect $l(\square)$ for $\square$ the square containing $e\cap l(\sigma)$.

\begin{lem}
\label{lem:strips3}
Let $e$ be an edge and $\rho$ an $e$-translation. If $\heartsuit$ holds for all $12$-gons~$\gamma$ in one $\rho$-orbit, then it holds for all $\gamma$.
\end{lem}

The proof is easy, it goes along the same lines as the proofs of Lemmas~\ref{lem:strips} and~\ref{lem:strips2} and we omit it.

\begin{proof}[Proof of Proposition~\ref{prop:236}]
We adopt the convention that if we label the vertices of an edge in a 12-gon~$\sigma$ by $v_0v_1$, then all the other vertices of $\sigma$ get cyclically labelled by $v_2\cdots v_{11}$.

Let $\tau$ be a 12-gon and suppose that $l(\tau)$ contains a vertex $v_1$ of an edge $v_0v_1\subset \tau$ for a square $\square=v_0v_1u_1u_0$. Let $\sigma$ be the 12-gon containing $u_0u_1$. Then $l(\square)=u_1v_0$ and furthermore $l$ assigns to the hexagon and square containing $u_1u_2,u_2u_3$, respectively, the longest diagonal containing $u_2,u_3$, respectively. Thus the only three remaining options for $l(\sigma)$ are the diagonals $u_0u_6, u_5u_{11}$, and $u_4u_{10}$. Thus we fall in one of the following three cases.

\smallskip

\noindent\textbf{Case 0.} There is such a $\tau$ with $l(\sigma)=u_4u_{10}$.

\smallskip

Let $\phi$ be the hexagon containing $u_3u_4$. Then $l(\phi)$ is parallel to $u_3u_4$ and there is no admissible choice for $l$ in the square containing $u_4u_5$ (see Figure \ref{Figure_lox_4}). This is a contradiction.

\begin{figure}[H]
\begin{center}
 \scalebox{1}{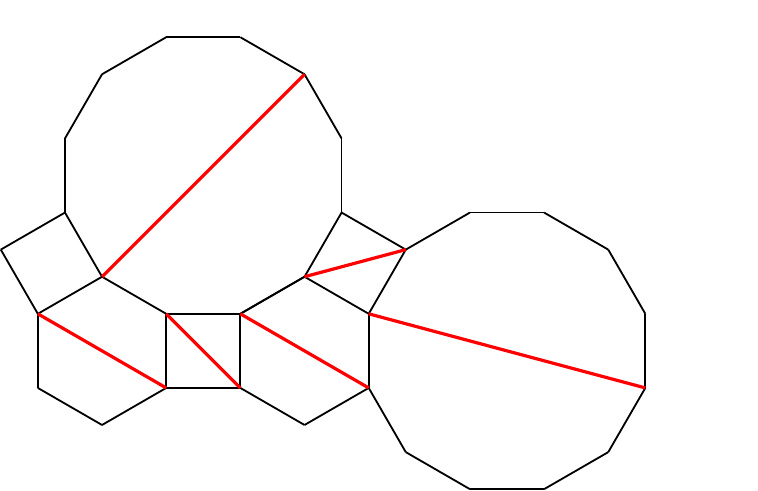}
\caption{}
\label{Figure_lox_4}
\end{center}
\end{figure}

\noindent\textbf{Case 1.} There is such a $\tau$ with $l(\sigma)=u_5u_{11}$.

\smallskip

It is easy to see that $l$ agrees with Figure \ref{Figure_lox_5} on the hexagon containing $v_0v_{11}$ and the square containing $v_{11}v_{10}$. Thus the only $2$-cell $\phi$ with $l(\phi)$ containing $v_{10}$ may be (and is) the hexagon containing $v_{10}v_9$. Consequently the only $2$-cell $\blacksquare$ with $l(\blacksquare)$ containing $v_{9}$ may be (and is) the square containing $v_{9}v_8$. Denote by $\sigma'$ the 12-gon adjacent to both $\phi$ and $\blacksquare$ at the vertex $x\neq v_9$. Then $x$ may lie (and lies) only in~$l(\sigma')$. Denote by $\rho$ the translation mapping $\sigma$ to $\sigma'$.

\begin{figure}[H]
\begin{center}
 \scalebox{1}{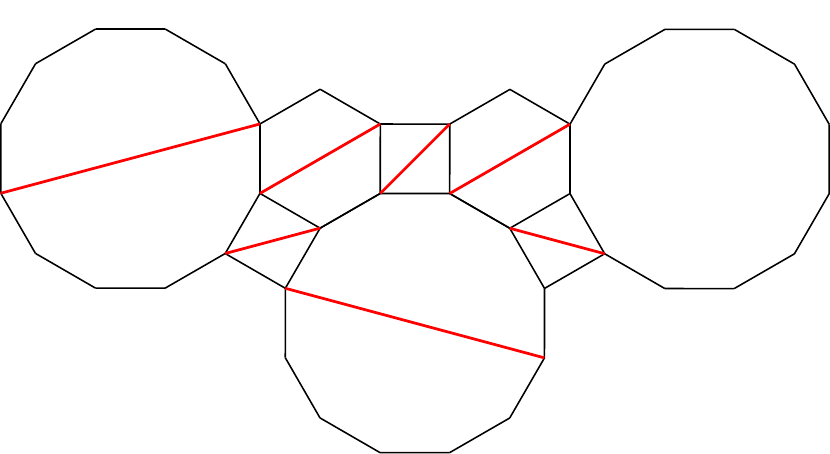}
\caption{}
\label{Figure_lox_5}
\end{center}
\end{figure}

It is also easy to see that the $2$-cells surrounding $\sigma$ and $\rho^{-1}(\tau)$ depicted in Figure~\ref{Figure_lox_6} have $l(\cdot)$ as indicated. This leaves only two choices for $l(\rho^{-1}(\tau))$, where one of them leads to Case~0, and the other satisfies $\heartsuit$.

\begin{figure}[H]
\begin{center}
 \scalebox{1}{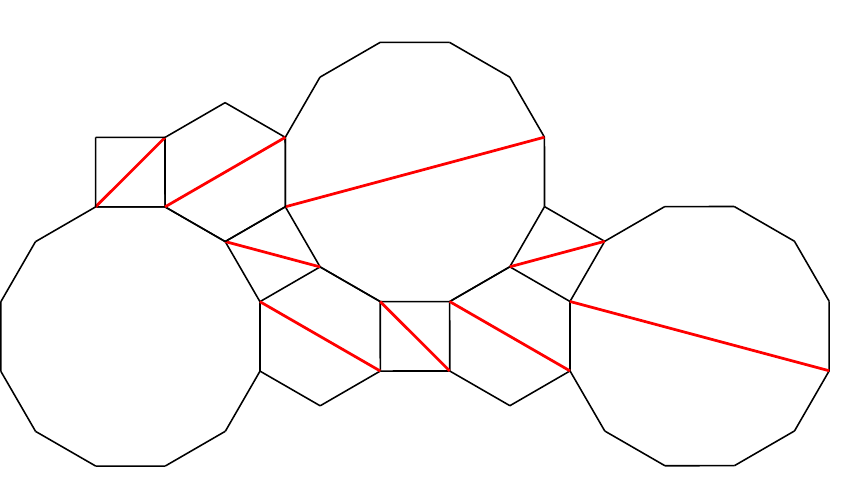}
\caption{}
\label{Figure_lox_6}
\end{center}
\end{figure}

Replacing repeatedly $\sigma$ and $\tau$ in the above argument by $\tau$ and $\sigma'$ or by $\rho^{-1}(\tau)$ and $\sigma$ gives that $\heartsuit$ holds for $\gamma=\rho^{n}(\sigma), \rho^{n}(\tau)$ for all $n\in \Z$. It remains to apply Lemma~\ref{lem:strips3}.

\smallskip
\noindent\textbf{Case 2.} There is no such $\tau$ as in Case~0 or~1.

\smallskip

It is then easy to see that $e=v_0v_1$ and $\rho$ mapping $\sigma$ to $\tau$ satisfy the hypothesis of Lemma~\ref{lem:strips3}.
\end{proof}

Note that for $l$ satisfying $\heartsuit$ for all 12-gons $\gamma$, the values of $l$ on 12-gons determine its values on squares and hexagons.

We are now ready for the following.

\begin{proof}[Proof of Theorem~\ref{thm:main}]
Let $F$ be a Euclidean plane isometrically embedded in $\Phi$ with a proper (and thus cocompact) action of $\Z^2$. By Lemma~\ref{lem:tiling}, the tiling of $F$ is either the standard square tiling, or the one of the Davis complex $\Sigma$ for $W$, where $W$ is the triangle Coxeter group with exponents $\{3,3,3\},\{2,4,4\}$ or $\{2,3,6\}$.

First consider the case where the tiling of $F$ is the standard square tiling. We can partition the set of edges into two classes \emph{horizontal} and \emph{vertical} of parallel edges. Let $T$ be the set of types of horizontal edges and $T'$ be the set of types of vertical edges. Since for each square the type of its two horizontal (respectively, vertical) edges is the same, we have $m_{tt'}=2$ for all $t\in T, t'\in T'$. Moreover, by Remark~\ref{rem:AS}, if one of the edges is of form $g,gt^k$, then the other is of form $h,ht^k$. Thus, up to a conjugation, the stabiliser of $F$ in $A_S$ is generated by a horizontal translation $w\in A_T$ and a vertical translation $w'\in A_{T'}$. This brings us to Case~(a) in Theorem~\ref{thm:main}.

It remains to consider the case where the tiling of $F$ is the one of $\Sigma$. Consider its induced polarisation $l$ from Definition~\ref{def:induced_polarisation}. By Lemma~\ref{lem:adm}, $l$ is admissible. By Propositions~\ref{prop:333}, \ref{prop:244}, and~\ref{prop:236}, there is an edge $e$ such that each for each $\gamma$ a maximal size $2$-cell, the diagonal $l(\gamma)$ has endpoints on edges of $\gamma$ parallel to $e$, and there is a particular translation $\rho$ in the direction perpendicular to $e$ preserving $l$.

For an edge $f$ of type $\{s\}$ in $\Sigma$, its vertices are of form $g,gs^k$ for $k>0$, directed from $g$ to $gs^k$. If $k>1$, then we call $f$ \emph{$k$-long}. By Lemma~\ref{lem:AS2} and Remark~\ref{rem:AS}, if $f$ is $k$-long, then so is its opposite edge in both of the $2$-cells that contain $f$. Consequently all the edges crossing the bisector of $f$ are $k$-long. Moreover, all such bisectors are parallel, since otherwise the $2$-cell $\square$ where they crossed would have four long edges, so $\square$ would be a square by Lemma~\ref{lem:AS2}. Analysing $l$ in the 2-cells adjacent to $\square$ leaves then no admissible choice for $l(\square)$.

Furthermore, by Lemmas~\ref{lem:AS2},~\ref{lem:strips},~\ref{lem:strips2} and~\ref{lem:strips3}, if $f$ is a long edge, then we can assume that $f$ is parallel to $e$.

Suppose first that $W_T$ is the triangle Coxeter group with exponents $\{3,3,3\}$ and $T=\{s,t,r\}$. Let $\omega$ be a combinatorial axis for the action of $\rho$ on $\Sigma$. Since none of the edges of $\omega$ are parallel to $e$, by the definition of the induced polarisation we see that they are all directed consistently (see Figure~\ref{Figure_lox_7}).

\begin{figure}[H]
\begin{center}
 \scalebox{0.8}{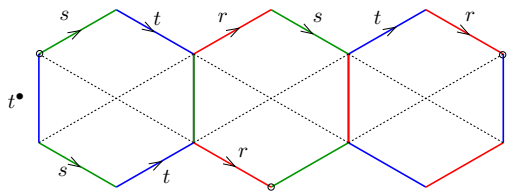}
\caption{Dashed lines indicate possible polarisations.}
\label{Figure_lox_7}
\end{center}
\end{figure}

Thus, up to replacing $F$ by its translate and interchanging $t$ with $r$, the element $strstr$ preserves $\omega$, and coincides on it with $\rho^3$. In fact, $\rho^3$ not only preserves the types of edges, but also by Lemma~\ref{lem:AS2} their direction and $k$-longness. Thus $strstr$ preserves $F$. The second generator of the type preserving translation group of $\Sigma$ in $W_T$ is $tstr$. Note that the path representing it in $\Sigma\subset \Phi$ corresponds to a word in $t^\bullet str\subset A_T$. That word depends on whether the second edge of the path is long and on the polarisation. Since $\Z^2$ acts cocompactly, there is a power of $tstr$ such that its corresponding path in $\Sigma\subset \Phi$ reads off a word in $(t^\bullet str)^*\subset A_T$ that is the other generator of the orientation preserving stabiliser of $F$ in $A_S$. This brings us to Case~(b) in Theorem~\ref{thm:main}. One similarly obtains the characterisations of orientation preserving stabilisers of $F$ for the two other $W$ (see Figures \ref{Figure_lox_8} and \ref{Figure_lox_9}).

\begin{figure}[H]
\begin{center}
 \scalebox{0.9}{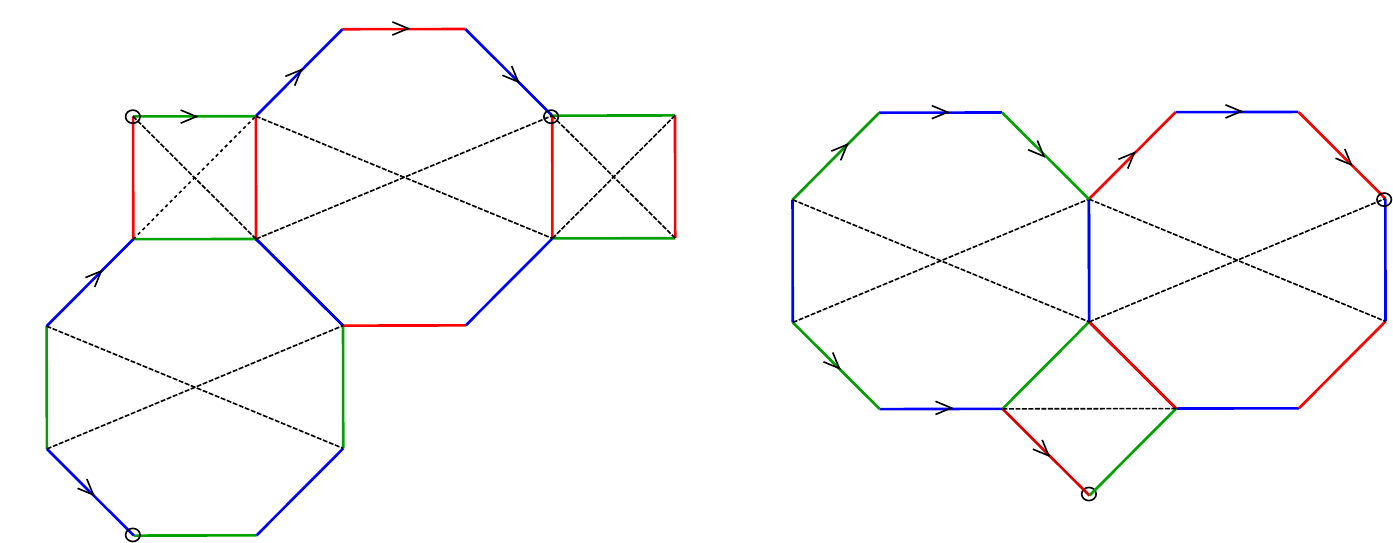}
\caption{}
\label{Figure_lox_8}
\end{center}
\end{figure}

\begin{figure}[H]
\begin{center}
 \scalebox{0.9}{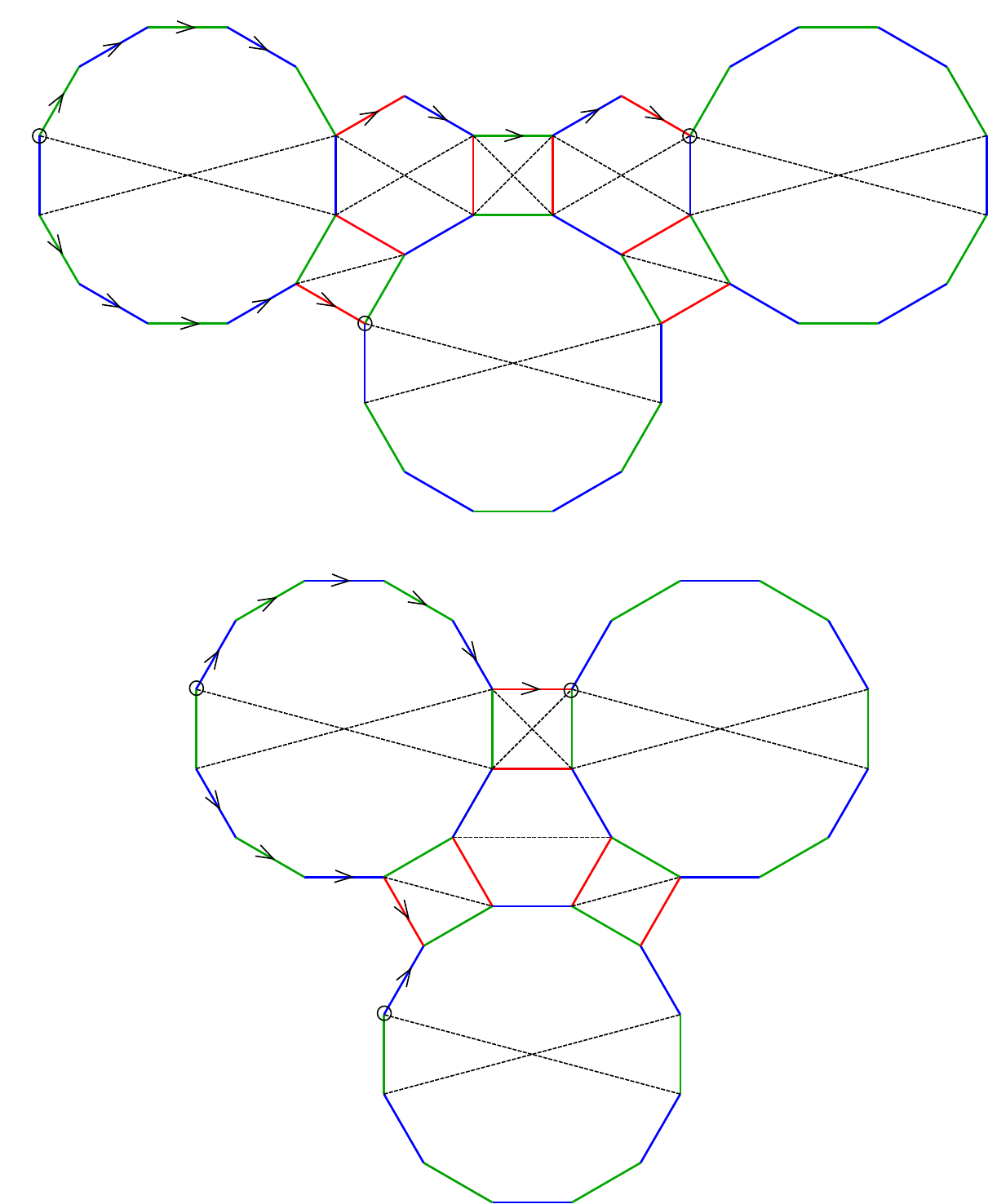}
\caption{}
\label{Figure_lox_9}
\end{center}
\end{figure}
\end{proof}

\begin{rem}
\label{rem:klein}
Analysing the full stabilisers of $F$ in $A_S$ one can easily classify also the subgroups of $A_S$ acting properly on $\Phi$ isomorphic to the fundamental group of the Klein bottle. For example, suppose that the second generator of the $\Z^2$ in Case~(b) of Theorem~\ref{thm:main} has the form
$g=(t^kstr)(t^{-k}str)$. Then our $\Z^2$ is generated by $strstr$ and $g'=g(strstr)^{-1}=t^ksr^{-k}s^{-1}$. Note that $str$ normalises our $\Z^2$ with
$(str)^{-1}g'(str)=(g')^{-1}$. Thus $\langle str, g'\rangle$ is isomorphic to the fundamental group of the Klein bottle. We do not include a full classification, since it is not particularly illuminating.
\end{rem}

\end{document}